\documentclass[a4paper]{amsart}

\usepackage{amsmath}
\usepackage{mathtools}
\usepackage{amsthm}
\usepackage{enumerate}
\usepackage[varg]{txfonts} 
\usepackage[arrow, curve, matrix, all]{xy}
\usepackage[dvipdfmx]{graphicx}
\usepackage{siunitx}
\usepackage{xcolor}
\usepackage{vertbars}
\usepackage{tikz}

\allowdisplaybreaks

\usepackage{mathrsfs}

\usepackage{mathcomp}

\theoremstyle{definition}
\newtheorem{dfn}{\textbf{Definition}}[section]

\newtheorem{ex}[dfn]{\textbf{Example}}
\newtheorem{prop}[dfn]{\textbf{Proposition}}

\newtheorem{thm}[dfn]{\textbf{Theorem}}

\renewcommand{\proofname}{\Proof}
\makeatletter
\renewenvironment{proof}[1][\proofname]{\par
  \pushQED{\qed}%
  \normalfont \topsep6\p@\@plus6\p@\relax
  \trivlist
  \item\relax
  {
  #1\@addpunct{.}}\hspace\labelsep\ignorespaces
}{%
  \popQED\endtrivlist\@endpefalse
}
\makeatother

\renewcommand{\today}{\the\year/\the\month/\the\day}

\numberwithin{equation}{section}

\newcommand{\fS}{\mathfrak{S}}

\newcommand{\Z}{\mathbb{Z}}

\newcommand{\sgn}{\mathrm{sgn}}
\newcommand{\sign}{\mathrm{sign}}

\newcommand{\dokh}[4]{\epsilon(#1;{#2}, C_{#3}, C_{#4})}
\newcommand{\dokhplus}[4]{\epsilon(#1;{#2}, {#3}, {#4})}
\newcommand{\nikh}[3]{\delta(#1;C_{#2}, C_{#3})}

\newcommand{\domu}[3]{\mu(C_{#1}, C_{#2}, C_{#3})}
\newcommand{\mimu}[2]{\bar{\mu}_{#1}(#2)}

\newcommand{\vl}[2]{V(L_{#1}, L_{#2})}
\newcommand{\va}[1]{V_a(L_{#1})}
\newcommand{\vla}[2]{V_a(L_{#1}, L_{#2})}
\newcommand{\vlap}[2]{V_a(L_{#1}, L_{#2}, +)}
\newcommand{\vlam}[2]{V_a(L_{#1}, L_{#2}, -)}
\newcommand{\LN}[2]{\mathrm{Link}(L_{#1}, L_{#2})}
\newcommand{\eL}[3]{a_{L_{#1}}({#2}, {#3})}

\newcommand{\marubatu}[2]{\raisebox{-1.5pt}[0pt][0pt]{\scalebox{0.5}{#1}}
\raisebox{-1.0pt}[0pt][0pt]{\scalebox{0.04}{\includegraphics{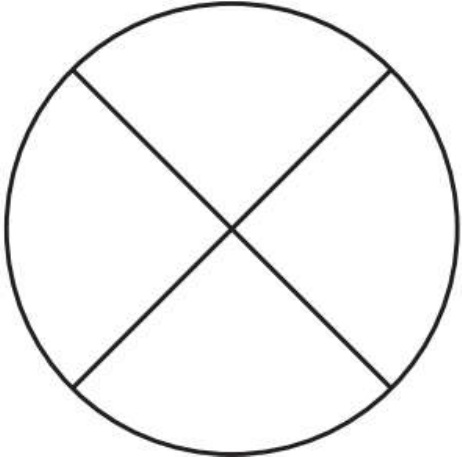}}}
\raisebox{-1.5pt}[0pt][0pt]{\scalebox{0.5}{#2}}\,}

\definecolor{Black}{cmyk}{0, 0, 0, 1}
\definecolor{Color1}{rgb}{0.9961,0.0000,0.0000}
\definecolor{Color2}{rgb}{0.0000,0.0000,0.9961}
\definecolor{Color3}{rgb}{0.0000,0.9961,0.0000}
\definecolor{Color4}{rgb}{0.9961,0.9961,0.0000}
\definecolor{Color5}{rgb}{0.9961,0.0000,0.9961}
\definecolor{Color6}{rgb}{0.0000,0.9961,0.9961}
\definecolor{White}{rgb}{1.0000,1.0000,1.0000}

\title{A representation of Milnor's triple linking number by chord diagrams and doodle invariants}
\author{RYOSUKE HIRATA}
\email{23ss108k@gmail.com}
\date{\today}
\begin{document}
  \begin{abstract}
    From \cite{Integral geometry of plane curves and knot invariants} it follows that degree two knot invariant admits a decomposition into the sum of a Gauss diagram count and a term involving Arnold invariants. 
    In this paper we establish an analogous description for Milnor's triple linking number — likewise of degree two — showing that it can be represented in terms of counts of certain chord diagrams together with doodle invariants. 
  \end{abstract}
  \maketitle
  \section{Introduction} 
  This paper discusses the relationship between doodle invariants and link invariants. 

  J. Milnor \cite{Isotopy of links} defined a series of link invariants known as the Milnor invariants. 
  The simplest invariant in this series is the linking number, and the second simplest is Milnor's triple linking number, which is a link homotopy invariant for three-component links. 

  This invariant is difficult to compute directly from its definition. 
  R. Fenn and P. Taylor \cite{Introducing doodles} showed that Milnor's triple linking number for special links, whose projection is a union of three simple closed curves with cyclic height (the Borromean ring is one such example), is equal to a doodle invariant, call the \emph{$\mu$-invariant}. 
  Here, a doodle refers to a union of generic closed curves on a plane and can be thought of as the projection of links onto the plane. 
  Two doodles are said to be equivalent if one can be transformed into the other by a finite sequence of permitted Reidemeister moves. 
  The invariant of the doodles we consider in this paper are those that remain unchanged under this equivalence. 

  The main result of this paper (Theorem \ref{syuteiri}) extends the result of \cite{Introducing doodles}. 
  Milnor's triple linking number for general three-component links is expressed as the sum of a doodle invariant and the sum of the signs of the intersections of chord diagrams, which are constructed through a certain procedure from the link. 
  This result is analogous to the work of X.-S. Lin and Z. Wang \cite{Integral geometry of plane curves and knot invariants}, that expresses the knot invariant of order two as the sum of an invariant of Gauss diagrams and the Arnold invariants for generic plane curves. 

  The proof of the main result is based on a result by B. Mellor and P. Melvin \cite{A geometric interpretation of Milnor’s triple linking numbers}. 
  They showed that Milnor's triple linking number can be expressed as the sum of the signs of the triple points of the Seifert surfaces bounded by each component of the link, as well as an integer derived from the intersection of the Seifert surfaces and the link. 
  In the proof, we demonstrate that by successfully constructing a Seifert surface from the chord diagrams we create, Milnor's triple linking number can be expressed solely in terms of the signed sum of the triple points of the Seifert surfaces. 
  Moreover, this sum can be computed by the invariants of the doodle and the chord diagrams. 

  \section{Doodles and the $\mu$-invariant}
\subsection{Doodles}
\begin{dfn}\label{n-component doodle}[cf.\ \cite{Introducing doodles}, \cite{On the classification of ornaments}] 
  An \emph{n-component doodle} is an immersion $\coprod_{i=1}^n S_i^1 \to \mathbb{R}^2$ that has only finitely many transversal double points as its singularities, where $S_i^1$ ($i=1, \dots, n$) denotes a copy of $S^1$. 
  We denote a doodle by $D = (C_1, \dots, C_n)$, where each $C_i$ is the image of $S_i^1$. We call each $C_i$ a \emph{component} of the doodle. 
  When the number of components is not important, we simply refer to them as doodles. 
\end{dfn}
\begin{ex} 
  Figure \ref{figure_doodle_ex.tex} shows examples of doodles. 
  In Figure \ref{figure_doodle_ex.tex} (a) and (b), curves with different colors represent different components. 
    \begin{figure}[htbp]
      \centering
      \includegraphics[height=29mm]{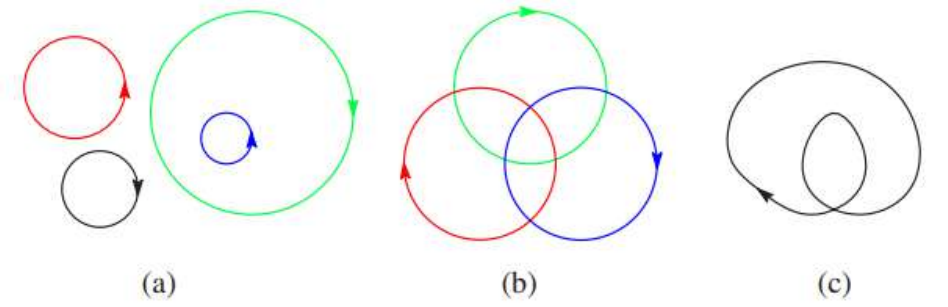}
      \caption{Examples of doodles}
      \label{figure_doodle_ex.tex}
    \end{figure}
\end{ex}
\begin{dfn}\label{doodleequivalent}[cf.\ \cite{Introducing doodles}, \cite{On the classification of ornaments}] 
  Two doodles are said to be \emph{equivalent} if and only if they can be transformed into each other by a finite sequence of ambient isotopies of $S^2$, RI (Reidemeister move I), RII (Reidemeister move II), and a limited form of RIII (Reidemeister move III) (see Figures \ref{figure_Reidemeistermove.tex} and \ref{figure_forbidden_move.tex}).
  \begin{figure}[htbp] 
    \centering
    \includegraphics[width=130mm]{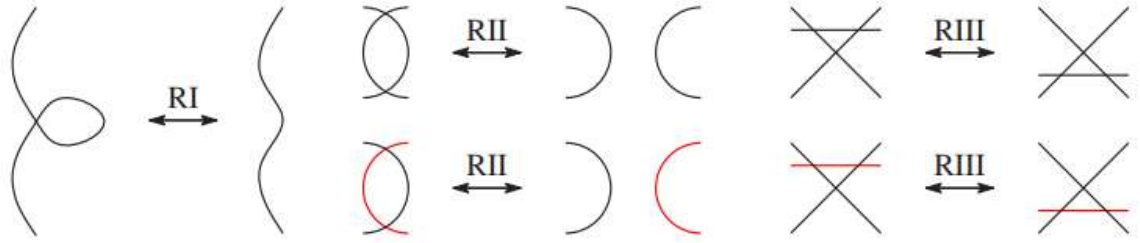}
    \caption{Permitted moves (each component may be arbitrarily oriented)} 
    \label{figure_Reidemeistermove.tex} 
  \end{figure} 
  \begin{figure}[htbp] 
    \centering 
    \includegraphics[width=60mm]{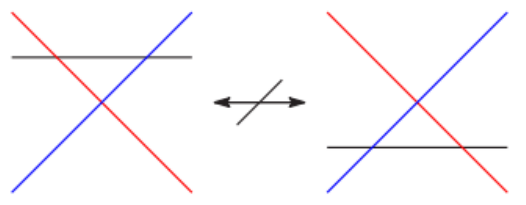} 
    \caption{Forbidden move} 
    \label{figure_forbidden_move.tex} 
  \end{figure} 
\end{dfn}

  \subsection{The $\mu$-invariant} 
  The $\mu$-invariant is defined in \cite{Introducing doodles} for three-component doodles, where all the components are simple closed curves with an anticlockwise orientation. 
  This definition can be extended to three-component doodles in the sense of Definition \ref{n-component doodle} as follows. 

  By smoothing all self-intersections of $C_1$, we obtain a doodle $(C_{1, 1}, \dots, C_{1, a})$, where all components are oriented simple closed curves (see Figure \ref{figure_smoothing.tex}). 
  \begin{figure}[htbp] 
    \centering 
    \includegraphics[width=120mm]{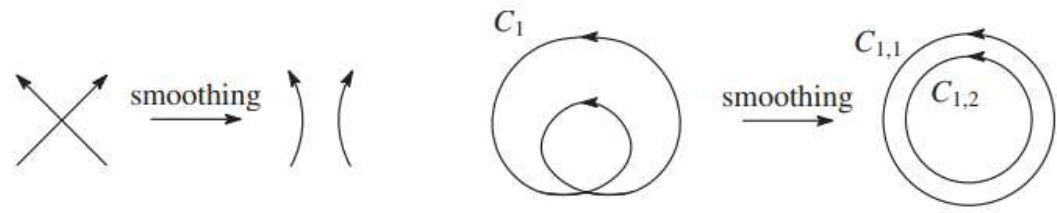} 
    \caption{Smoothing} 
    \label{figure_smoothing.tex} 
  \end{figure}

  For each intersection $p$ of $C_2$ and $C_3$, consider the sign $\nikh{p}{2}{3}$. 
  This sign is defined such that $\nikh{p}{2}{3}=+1$ when $C_3$ intersects from right to left,  and $\nikh{p}{2}{3}=-1$ when it intersects from left to right, with respect to the direction of travel of $C_2$. 
  For each intersection $p$ of $C_2$ and $C_3$ in the bounded region whose boundary is $C_{1, i}$, we define $\dokh{p}{C_{1, i}}{2}{3}$ as 
  \begin{equation} 
    \dokh{p}{C_{1, i}}{2}{3}=(-1)^{n_{C_{1, i}}}\nikh{p}{2}{3}, 
  \end{equation} 
  where $n_{C_{1, i}}$ is 0 if $C_{1, i}$ is oriented anticlockwise, and 1 if $C_{1, i}$ is oriented clockwise. 

\begin{dfn}\label{mu-invariantdef} 
  Let $\domu{1, i}{2}{3}$ denote the sum of $\dokh{p}{C_{1, i}}{2}{3}$ for all intersections of $C_2$ and $C_3$ inside $C_{1, i}$.
  The \emph{$\mu$-invariant} of the doodle $D$ is then defined as 
    \begin{equation} 
      \domu{1}{2}{3}:=\sum_{i=1}^a \domu{1, i}{2}{3}. 
    \end{equation} 
\end{dfn}
\begin{ex} 
  Figure \ref{figure_exdoodlemuinv.tex} shows examples of calculation of the $\mu$-invariant.
  \begin{figure}[htbp] 
    \centering 
    \includegraphics[width=95mm]{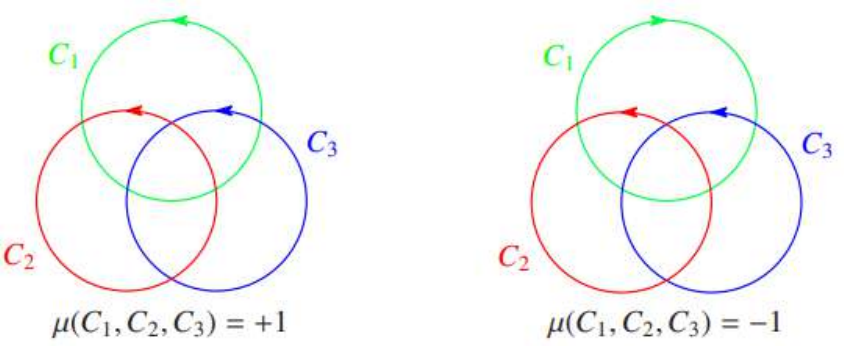} 
    \caption{Examples of $\mu$-invariants} 
    \label{figure_exdoodlemuinv.tex} 
  \end{figure} 
\end{ex}
\begin{prop}\label{mu-invariantprop} 
  \begin{enumerate}[(1)] 
    \item 
      $\domu{1}{2}{3}\in\Z$. 
    \item 
      Let $-C_i$ denote the curve $C_i$ with its orientation reversed. 
      Then, we have 
        \begin{equation}\label{mucomponentmukihantai} 
          \mu((-1)^{s_1}C_1, (-1)^{s_2}C_2, (-1)^{s_3}C_3)=(-1)^s\domu{1}{2}{3}\quad (s_1, s_2, s_3\in{0, 1}, s=s_1+s_2+s_3). 
        \end{equation}
    \item 
      Let $\fS_3$ denote the symmetric group on three elements. 
      Then, we have 
      \begin{equation}\label{mucomponentirekae} 
        \domu{\sigma(1)}{\sigma(2)}{\sigma(3)}=\sgn(\sigma)\domu{1}{2}{3}\quad (\sigma\in\fS_3). 
      \end{equation} 
  \end{enumerate} 
\end{prop} 
\begin{prop} 
  The $\mu$-invariant is invariant under the transformations shown in Figure \ref{figure_Reidemeistermove.tex}.
\end{prop} 
\begin{prop}\label{muinvariant1henka} 
  The $\mu$-invariant changes by $\pm1$ before and after the RIII move in Figure \ref{figure_forbidden_move.tex}. 
\end{prop}

  \section{Milnor's triple linking number} 
  The Milnor's triple linking number $\mimu{L}{ijk}$ of three-component links is introduced in \cite{Isotopy of links}. 
  \cite{A geometric interpretation of Milnor’s triple linking numbers} provides a formula for the triple linking number using two integers, $m_{ijk}$ and $t_{ijk}$, defined as follows. 

\begin{dfn}[{\cite[Definition 5.1]{Surface Systems and Triple Linking Numbers}}] 
  Let $S$ be a finite set. 
  Let $w = s_1^{\epsilon_1}\ldots s_m^{\epsilon_m}$ be a word in the letters $s_i^{\epsilon_i} \in S \times \{\pm 1\}$, and let $r, s \in S$. 
  The $rs$-\emph{decomposition} $(i, j)$ of $w$ is a pair of indices $i < j$ such that $s_i = r^{\pm 1}$ and $s_j = s^{\pm 1}$. 
  The set of $rs$-decompositions is written as $D_{rs}(w)$. 
  The \emph{sign} of a decomposition is $\sign_w(i, j) = \epsilon_i \cdot \epsilon_j \in {\pm 1}$. 
  The \emph{signed occurrence} $e_{rs}$ of the pair $(r, s)$ is the integer 
    \begin{equation} 
      e_{rs}(w) = \sum_{(i, j) \in D_{rs}(w)} \sign_w(i, j). 
    \end{equation} 
\end{dfn}
  Let $L$ be a link, and attach Seifert surfaces $F_i$ to the $i$-th component $L_i$ of the link $L$. 
  These surfaces are oriented according to the outward-normal-first convention. 
  Choose a base point on the component $L_1$. 
  Define a word $w_1$ on the set $S=\{2, 3\}$ as follows. 
  Starting from the base point, go around $L_1$ in the direction of $L_1$, and place $2^\epsilon$ (resp. $3^\epsilon$) each time we intersect $F_2$ (resp. $F_3$). 
  Here, $\epsilon$ is $+1$ if the orientations of the surface and $L_1$ are coincident at the intersection, and $-1$ otherwise. 
  Read the words $w_2$ and $w_3$ in the same way. 
\begin{dfn}
  Define $m_{123}(F)$ as 
    \begin{equation} 
      m_{123}(F)=e_{23}(w_1)+e_{31}(w_2)+e_{12}(w_3). 
    \end{equation}
\end{dfn}

\begin{dfn}\label{triplepointsign} 
  Let $n_i$ be the normal vector of the surface $F_i$, and let $F = F_i \cup F_j \cup F_k$. 
  Define $t_{ijk}(F)$ as 
  \begin{equation} 
    \begin{array}{ll} 
      t_{ijk}(F)\coloneqq\sum_{p \in F_i \cap F_j \cap F_k} \tau_{ijk}(p), \quad 
      &\tau_{ijk}(p)\coloneqq +1, \text{if $(n_i, n_j, n_k)$ is right-handed}, \\
      &\tau_{ijk}(p)\coloneqq -1, \text{if $(n_i, n_j, n_k)$ is left-handed}. 
    \end{array} 
  \end{equation} 
\end{dfn} 

  Define 
    \begin{equation} 
      \Delta_L(ijk):=\gcd\{\LN{i}{j}, \LN{j}{k}, \LN{k}{i}\}. 
    \end{equation} 
  By definition, 
    \begin{equation} 
      \gcd\{0, 0, 0\}=0. 
    \end{equation}

\begin{thm}[cf.\ {\cite[Theorem 5.14]{Surface Systems and Triple Linking Numbers}}, {\cite[Theorem(1)]{A geometric interpretation of Milnor’s triple linking numbers}}]\label{mimumtthm} 
  For any choice of Seifert surfaces $F$, the Milnor's triple linking number satisfies the congruence 
  \begin{equation}\label{Milnorkousiki} 
    \mimu{L}{ijk} \equiv m_{ijk}(F) - t_{ijk}(F) \mod \Delta_L(ijk). 
  \end{equation} 
\end{thm}

  \section{The main theorem}
\begin{dfn}
  A \emph{generalized chord diagram} is a chord diagram on a circle (called the \emph{skeleton}) that may have a vertex not connected by a chord to any other vertex, or a vertex with two or more chords emanating from it. 
  Such vertices are called \emph{excessive}. 
\end{dfn} 
  Let $L=L_1\cup L_2\cup L_3$ be a three-component link. 
  The procedure for constructing generalized chord diagrams from $L$ is as follows: choose a diagram of $L$ whose projection onto a plane is a doodle. 
  We now examine specific examples based on the Link $L=(L_1, L_2, L_3)$ shown in Figure \ref{figure_gutairei_LDLLDL.tex}.  
  \begin{figure}[htbp]
      \centering
      \includegraphics[width=170mm]{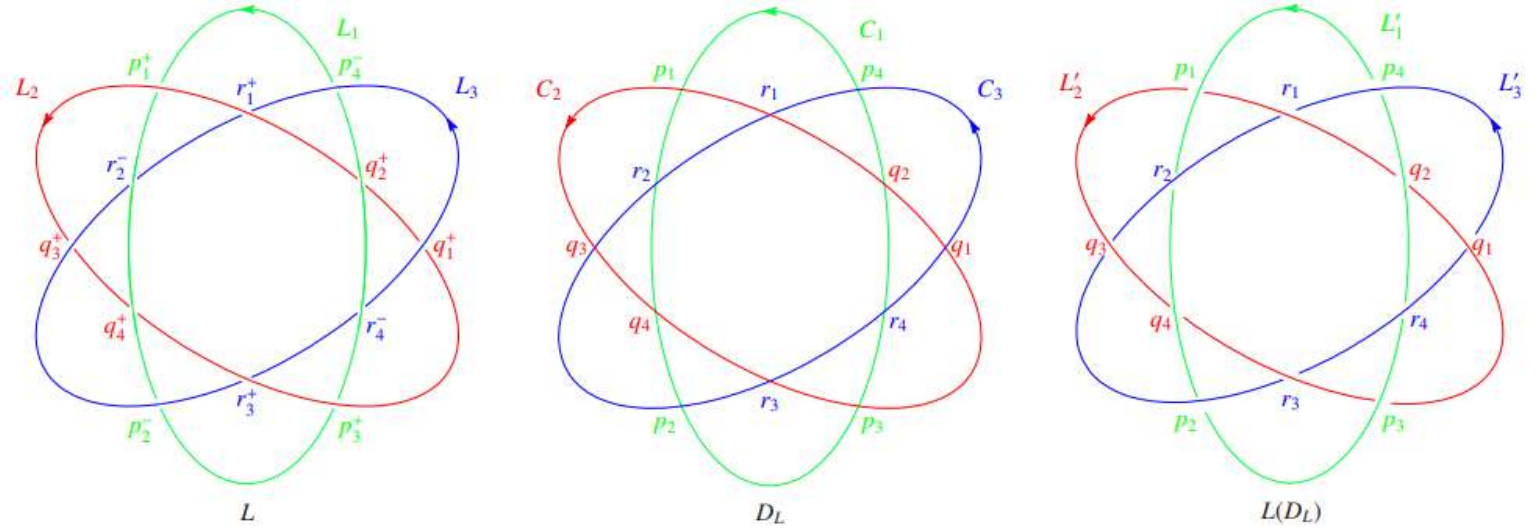}
      \caption{Link $L$, doodle $D_L$, and the link $L(D)$ constructed from $D_L$}
      \label{figure_gutairei_LDLLDL.tex}
  \end{figure}
  Next, consider a doodle $D_L=(C_1, C_2, C_3)$ obtained by projecting $L$ onto a plane.
  In \cite{Introducing doodles}, a 3-component link $L(D)=(L_1, L_2, L_3)$ is constructed from a 3-component doodle $D=(C_1, C_2, C_3)$ without self-intersections by specifying over/under arc information, such that $C_1$ is positioned over $C_2$, $C_2$ is positioned over $C_3$, and $C_3$ is positioned over $C_1$ at each intersection. 
  The construction can be easily extended to any doodle with self-intersections, yielding a link homotopy type $L(D)$. 
  The resulting $L(D)$ has the property that any two components are unlinked. 
  In particular, $\LN{i}{j}=0$. 
  The link $L(D_L)=(L^\prime_1, L^\prime_2, L^\prime_3)$ corresponding to this projection. 
  In Figure \ref{figure_gutairei_LDLLDL.tex}, $D_L$ and $L(D_L)$ are made of link $L$. 

  Let $\va{1}$ be the set of all crossings involving $L_1$ and another component of $L$ where the branch of $L_1$ is over-crossing. 
  Let $\gamma(p)$ be the sign of $p\in\va{1}$. 
  Let $\vl{i}{j}$ be the set of all intersections involving $L_i$ and $L_j$, let $\vla{i}{j}$ be the set of intersections where $L_i$ is the upper intersection. 
  Then we have
  \begin{equation}
    \va{1}=\vla{1}{2}\sqcup\vla{1}{3}. 
  \end{equation}
  For example, $\va{1}$ for $L$ in Figure \ref{figure_gutairei_LDLLDL.tex} can be divided as follows: 
  \begin{equation}
    \va{1}=\{q^+_2, r^-_2, q^+_4, r^-_4\}=\{q^+_2, q^+_4\}\sqcup\{r^-_2, r^-_4\}=\vla{1}{2}\sqcup\vla{1}{3}. 
  \end{equation}
  
  We take two points $b_{nc_2}$ and $b_{c_3}$ on $L_1$ so that no elements of $\va{1}$ are on the subarc of $L_1$ whose starting point is $b_{nc_2}$ and endpoint is $b_{c_3}$. 
  The next step is to construct a generalized chord diagram $G_{L_1}$ from $\va{1}$, $b_{c_3}$, and $b_{nc_2}$. 
  \begin{enumerate}
    \item
      Let $S_1$ represent the skeleton of $G_{L_1}$, which is oriented in correspondence with $L_1$. 
      On $S_1$, the elements of $\va{1}$, along with $b_{c_3}$ and $b_{nc_2}$, are placed in the corresponding order. 
    \item
      Connect each element $p$ of $\vla{1}{3}$ with $b_{c_3}$ by a chord. 
      These chords should be oriented so that its starting point (resp. endpoint) is $p$ if $p\in\vlap{1}{3}$ (resp. $p\in\vlam{1}{3}$). 
      Let $T_{1, 3}$ denote the set of all these chords. 
      Assume that no element of $T_{1, 3}$ intersects with another element of $T_{1, 3}$ along the way. 
    \item 
      Similarly, connect the element $\vla{1}{2}$ and $b_{nc_2}$ with an oriented chord. 
      Let $T_{1, 2}$ be the set of these chords. 
  \end{enumerate}

  We construct $G_{L_2}$ and $G_{L_3}$ in the same way as above. 
  We have $\va{k}=\vla{k}{j}\sqcup\vla{k}{i}$ for any cyclic permutation $(i, j, k)$ of $(1, 2, 3)$. 

\begin{dfn}
  Let $\left\langle\marubatu{j}{i}, G_{L_k}\right\rangle$ denote the sum of $\dokhplus{p}{S_k}{T_{k, j}}{T_{k, i}}$ for all intersections $p$ between $T_{k, j}$ and $T_{k, i}$. 
\end{dfn}

\begin{thm}\label{syuteiri}
  Let $L=L_1\cup L_2 \cup L_3$ be an oriented three-component link.
  Let $D_L=(C_1, C_2, C_3)$ be the three-component doodle obtained by projecting $L$ onto a plane. 
  Also, let $\mathfrak{A}_3$ be the alternating group on three elements. 
  Then, we have the following equation: 
  \begin{equation}\label{syuteirinokousiki}
    \mimu{L}{123}\equiv-\domu{1}{2}{3}-\sum_{(i, j, k)\in \mathfrak{A}_3}\left\langle\marubatu{j}{i}, G_{L_k}\right\rangle\mod\Delta_L(123). 
  \end{equation}
\end{thm}
\begin{ex}
  Let us calculate the triple linking number $\mimu{L}{123}$ for the link $L$ in Figure \ref{figure_gutairei_LDLLDL.tex}. 
  By computing each linking number, we obtain
  \begin{equation}
    \LN{1}{2}=2,\quad \LN{2}{3}=-2,\quad \LN{3}{1}=2,\quad \Delta_L(123)=2. 
  \end{equation}
  We also have $\domu{1}{2}{3}=2$ (see Figure \ref{figure_gutairei_LDLLDL.tex}). 
  Next, since $G_{L_1}$, $G_{L_2}$, and $G_{L_3}$ are as shown in Figure \ref{figure_GL_123.tex}, we have $\left\langle\marubatu{3}{2}, G_{L_1}\right\rangle=-3$, $\left\langle\marubatu{1}{3}, G_{L_2}\right\rangle=-1$, and $\left\langle\marubatu{2}{1}, G_{L_3}\right\rangle=-1$. 
  \begin{figure}[htbp]
      \includegraphics[width=150mm]{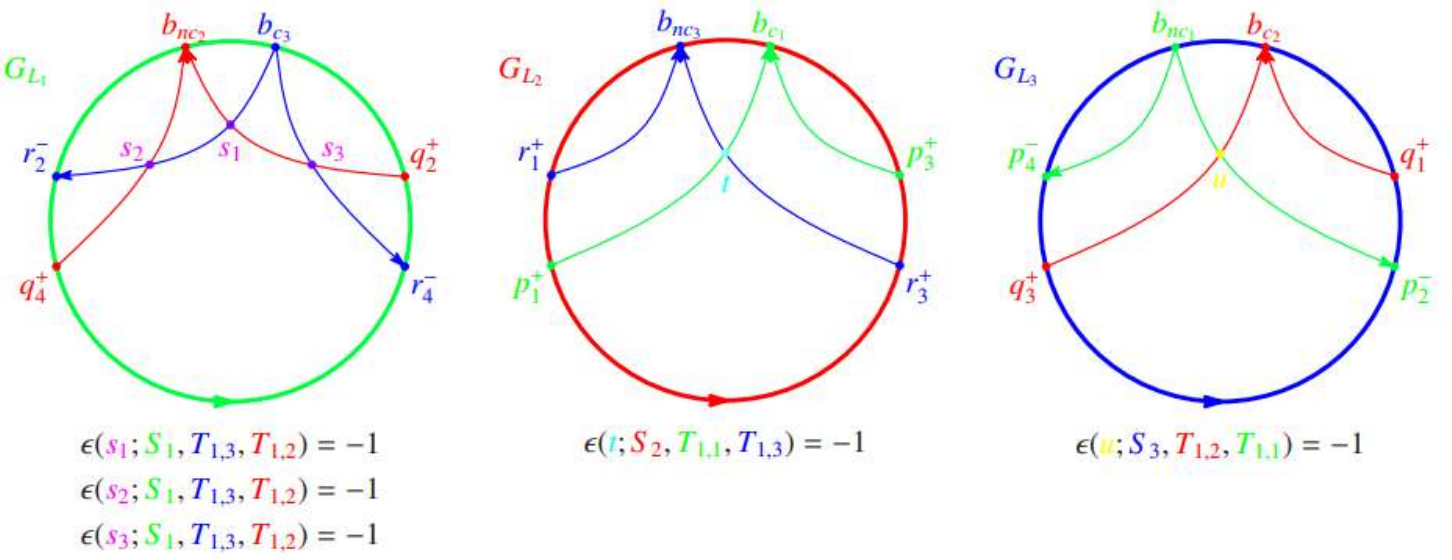}
      \caption{Generalized chord diagrams $G_{L_1}$, $G_{L_2}$, and $G_{L_3}$}
      \label{figure_GL_123.tex}
  \end{figure}
  From the above, we obtain: 
  \begin{equation}
      \mimu{L}{123}\equiv2-(-3-1-1)\equiv1\mod 2. 
  \end{equation}
\end{ex}
  To prove Theorem \ref{syuteiri}, we need the following theorem. 
\begin{thm}[cf.\ {\cite[Theorem3]{Introducing doodles}}]\label{IDthm3}
  Let $D=(C_1, C_2, C_3)$ be a three-component doodle.
  Then we have the following equation: 
  \begin{equation}
    \mimu{L(D)}{123}=-\domu{1}{2}{3}. 
  \end{equation}
\end{thm}
\begin{proof}[\emph{Proof}]
  We will consider three generalized chord diagrams $G_{L_k}\, (k=1, 2, 3)$, and based on these chord diagrams, we will consider a Seifert surface to link $L(D)=(L_1, L_2, L_3)$ for which $m_{123}(F)=0$. 
  
  First, we construct a connected Seifert surface for $L_2$ using the Seifert algorithm. 
  The disks to be attached should be like as shown in Figure \ref{figure_attach_disk.tex}. 
  \begin{figure}[htbp]
    \centering
    \includegraphics[width=110mm]{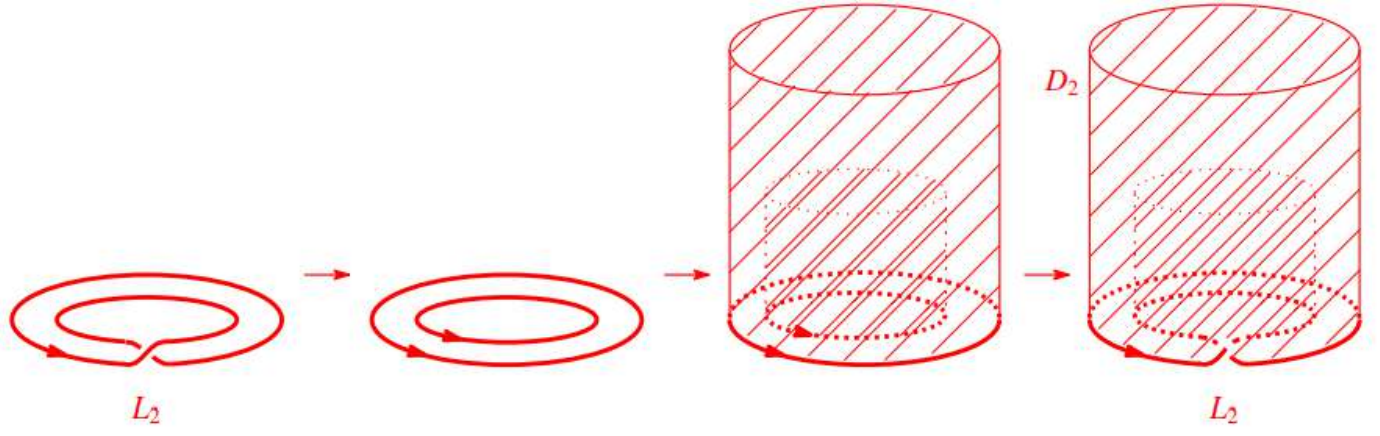}
    \caption{Attaching disks to obtain Seifert surface}
    \label{figure_attach_disk.tex}
  \end{figure}  
  Also, assume the attached disks are oriented according to the outward-normal-first convention. 
  Let $D_2$ be the surface constructed in this way. 
  
  Next, we ``stretch this surface $D_2$ as indicated by $G_{L_1}$''. 
  Stretching refers to the deformation shown in Figure \ref{figure_stretch_surface_D2.TEX}. 
  \begin{figure}[htbp]
    \centering
    \includegraphics[width=130mm]{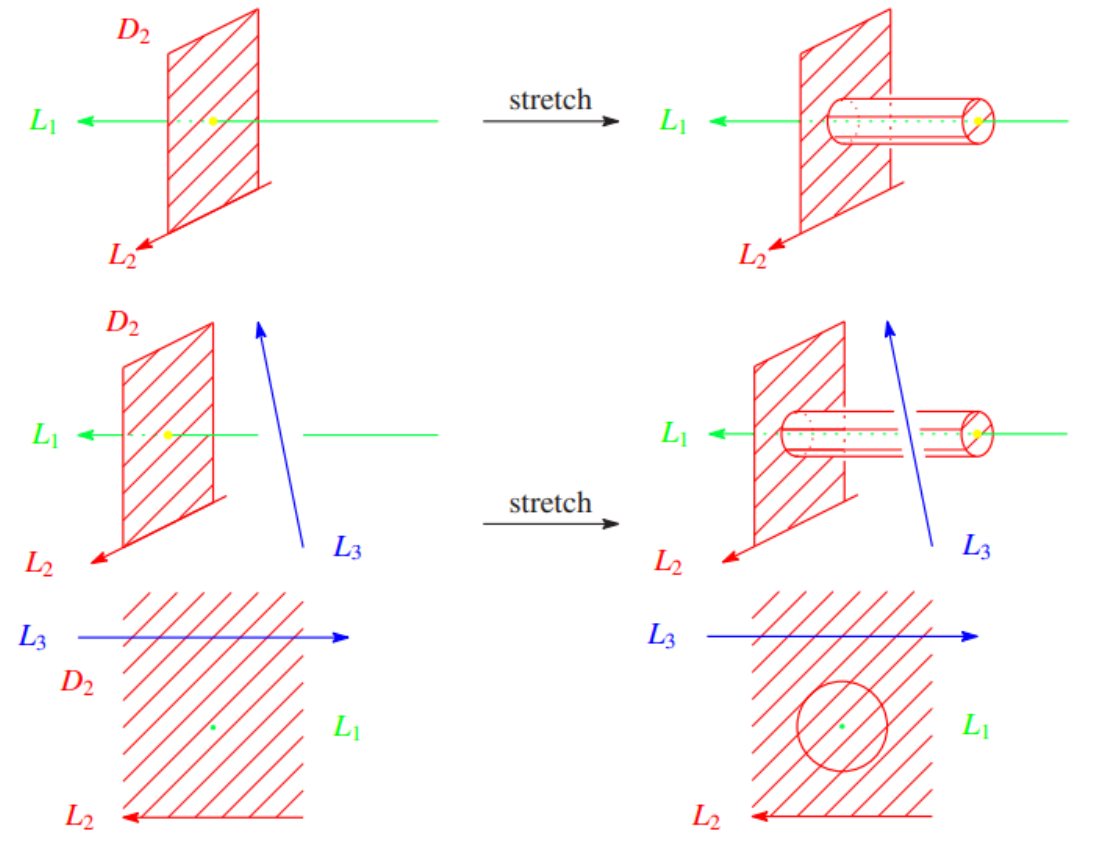}
    \caption{Stretching a Seifert surface}
    \label{figure_stretch_surface_D2.TEX}
  \end{figure}
  Stretch the surface from a neighborhood of the intersection of $D_2$ and $L_1$ in the opposite direction of $L_1$, until we arrive at $b_{nc_2}$. 
  At this point, the thickness and stretch destination should be slightly displaced to avoid self-intersection (see Figure \ref{figure_stretch_surface_bnc_2.TEX}). 
\begin{figure}[htbp]
  \centering
  \includegraphics[width=60mm]{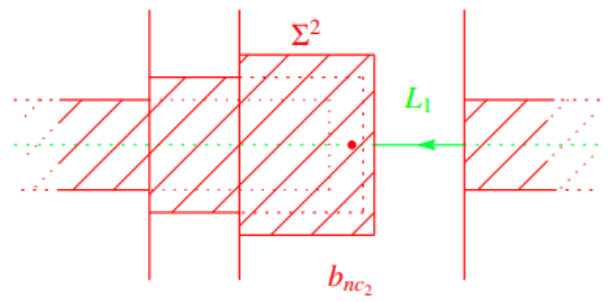}
  \caption{Configuration near $b_{nc_2}$}
  \label{figure_stretch_surface_bnc_2.TEX}
\end{figure}
  The resulting surface is denoted by $\Sigma^2$. 
  By constructing $D_1$ and $D_3$ and stretching them along $L_3$ and $L_2$, respectively, we obtain the Seifert surfaces $\Sigma^1$ and $\Sigma^3$ for $L_1$ and $L_3$. 
  However, when pushing up the disk attached to each $S^1$ in the process of the Seifert algorithm, $\Sigma^2$ is pushed higher than $\Sigma^1$, and $\Sigma^3$ is pushed higher than $\Sigma^2$ (see Figure \ref{figure_sigma123notakasawoarawasu.TEX}). 
  \begin{figure}[htbp]
    \centering
    \includegraphics[width=130mm]{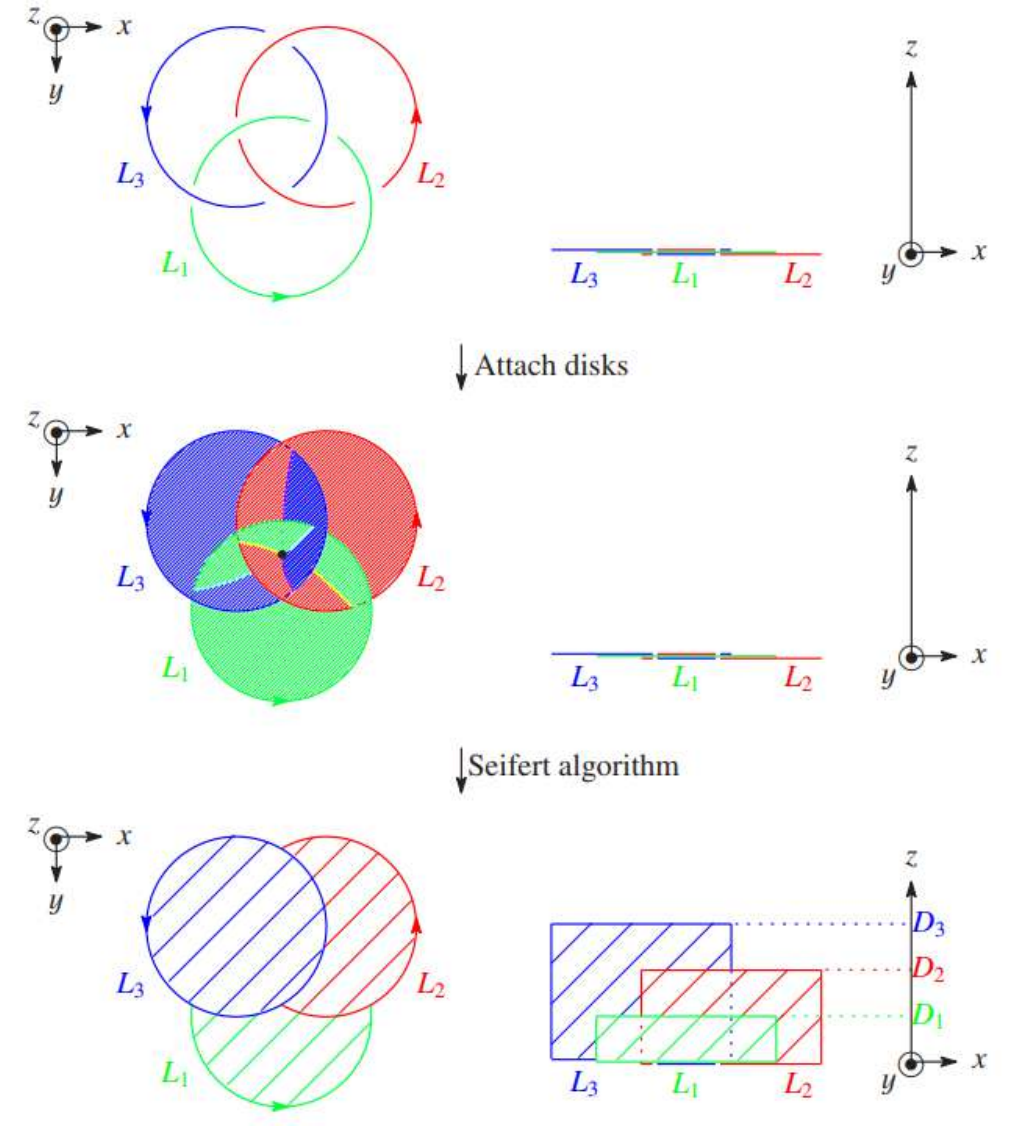}
    \caption{Relative heights of the Seifert surfaces}
    \label{figure_sigma123notakasawoarawasu.TEX}
  \end{figure}
  We will compute $m_{123}(\Sigma^1\cup\Sigma^2\cup\Sigma^3)$. 
  From the construction of $L(D)$, $L_1$ is below $L_3$, so $L_1$ does not intersect with $\Sigma^3$. 
  Thus, the word $w_1$ of the intersection points of $\Sigma^2$ with $L_1$ is 
  \begin{equation}
    w_1=2^\epsilon_1\cdots2^\epsilon_n \quad(\epsilon_i\in\{\pm1\}) 
  \end{equation}
  and hence $e_{23}(w_1)=0$. 
  Similarly, we find that $e_{31}(w_2)=0$ and $e_{12}(w_3)=0$, so 
  \begin{equation}
    m_{123}(\Sigma^1\cup\Sigma^2\cup\Sigma^3)=0. 
  \end{equation}
  We now compute $t_{123}(\Sigma^1\cup\Sigma^2\cup\Sigma^3)$. 
  The stretching from $D_i$ to $\Sigma^i$ do not change the set of triple points. 
  Therefore, 
  \begin{equation}
    t_{123}(\Sigma^1\cup\Sigma^2\cup\Sigma^3)=t_{123}(D_1\cup D_2\cup D_3). 
  \end{equation}
  Since $\Delta_{L(D)}(123)=0$, we obtain the following from \eqref{Milnorkousiki}: 
  \begin{equation}\label{eqtriplelinkingnumberto3zyuutensuu}
    \mimu{L(D)}{123}=-t_{123}(D_1\cup D_2\cup D_3). 
  \end{equation}
  From the construction of $D_1\cup D_2\cup D_3$, the triple points appear on the upper face of $D_1$ (see Figure \ref{figure_mainproof_t123mu123.tex}). 
  \begin{figure}[htbp]
    \centering
    \includegraphics[width=180mm]{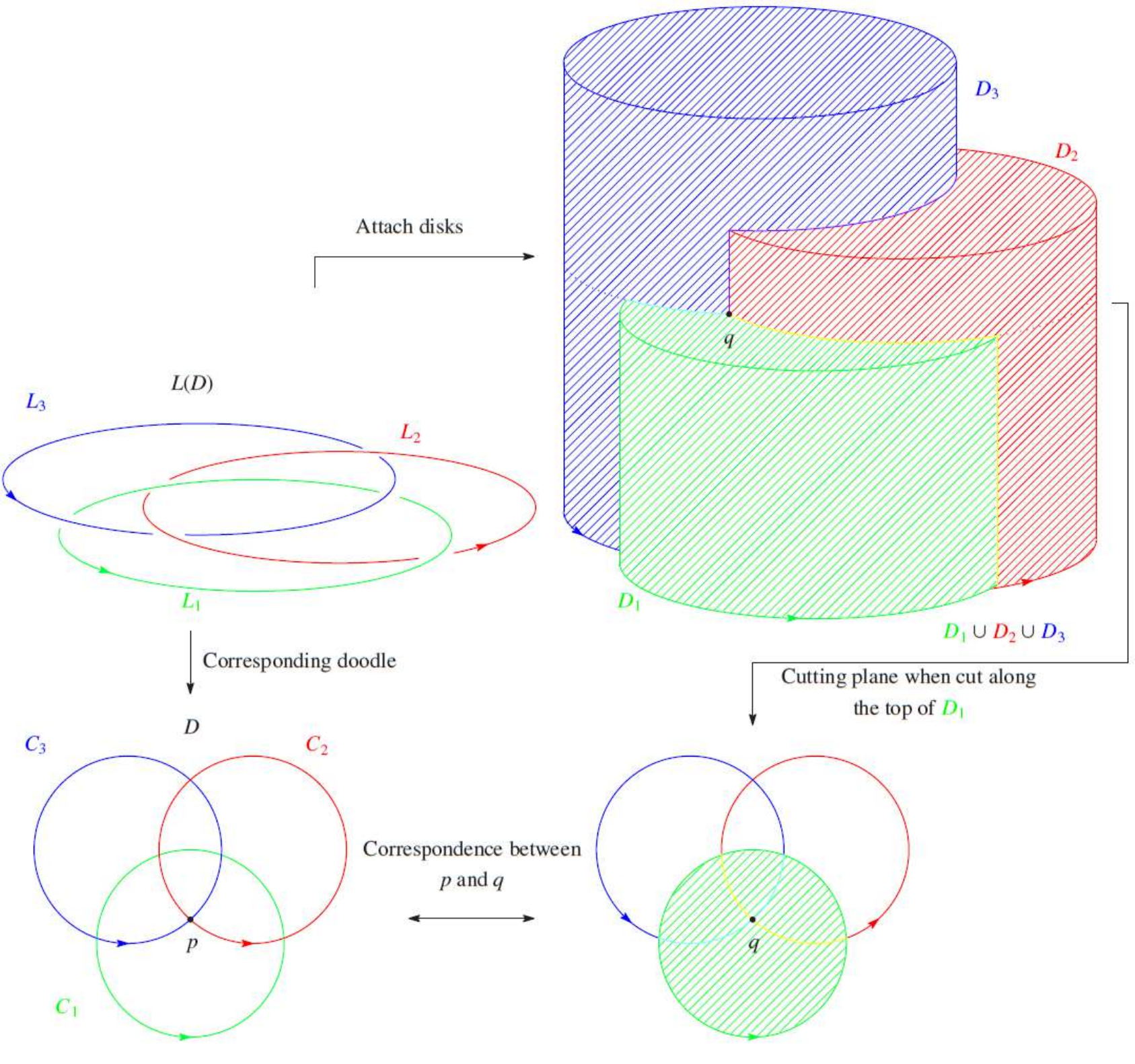}
    \caption{Correspondence between doodle intersections and triple points}
    \label{figure_mainproof_t123mu123.tex}
  \end{figure}
  The vertical line of $D_2\cap D_3$ corresponds to $C_2\cap C_3$, and the intersection of the vertical line and the upper face of $D_1$ is the triple point $q$, so $p$ and $q$ correspond. 
  We also have $\tau_{123}(q)=\dokh{p}{C_1}{2}{3}$ by inspection (see Figure \ref{figure_dokhtautaiou.TEX}). 
  \begin{figure}[htbp]
    \centering
    \includegraphics[width=150mm]{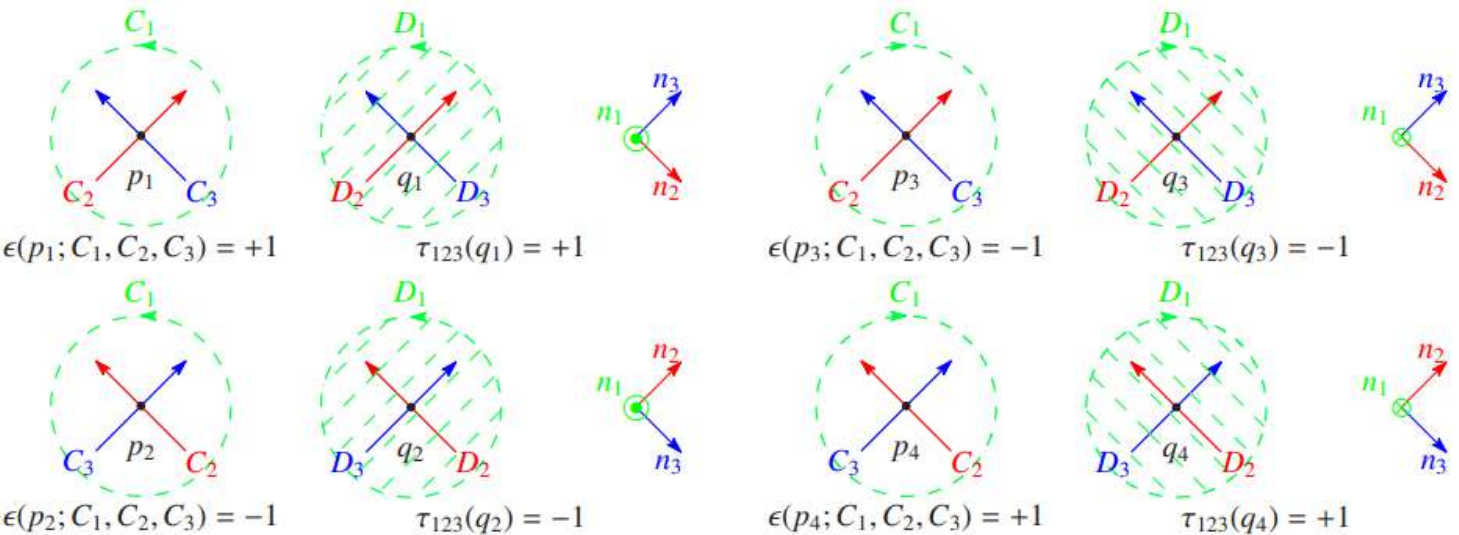}
    \caption{Correspondence between the signs of doodle intersections and triple points}
    \label{figure_dokhtautaiou.TEX}
  \end{figure}
  Therefore, 
  \begin{equation}\label{eq3zyuutensuutodoinv}
    t_{123}(D_1\cup D_2\cup D_3)=\domu{1}{2}{3}
  \end{equation}
  \eqref{eqtriplelinkingnumberto3zyuutensuu} and \eqref{eq3zyuutensuutodoinv} complete the proof. 
\end{proof}

  \subsection{Proof of main theorem}
\begin{proof}[\emph{Proof of Theorem \ref{syuteiri}}]
  First, we construct three generalized chord diagrams $G_{L_k} (k=1, 2, 3)$ corresponding to $L_k$. 
  We will attach the surface $D_k (k=1, 2, 3)$ to each $L_k$ in the same manner as in proof of Theorem \ref{IDthm3}. 
  We start by deforming $D_2$. 
  First, stretch the surface from the intersections $L_1\cap D_2$ corresponding to the elements of $\vla{1}{2}$, in the negative direction of $L_1$, until near $b_{nc_2}$.
  Then, stretch the surface from the intersections $L_3\cap D_2$ corresponding to the elements of $\vla{3}{2}$, in the positive direction of $L_3$, until near $b_{c_2}$. 
  The resulting surface is denoted by $\Sigma^2$. 
  Similarly, we perform the same transformations on $D_1$ and $D_3$, yielding the Seifert surfaces $\Sigma^1$ and $\Sigma^3$ for $L_1$ and $L_3$, respectively. 
  However, the parts of the surfaces that are stretched against the direction of each component must be thinner than the parts stretched along the direction, to avoid any intersection between the stretched parts (see Figure \ref{figure_stretch_surface_bbcc23.TEX}). 
  \begin{figure}[htbp]
    \centering
    \includegraphics[width=100mm]{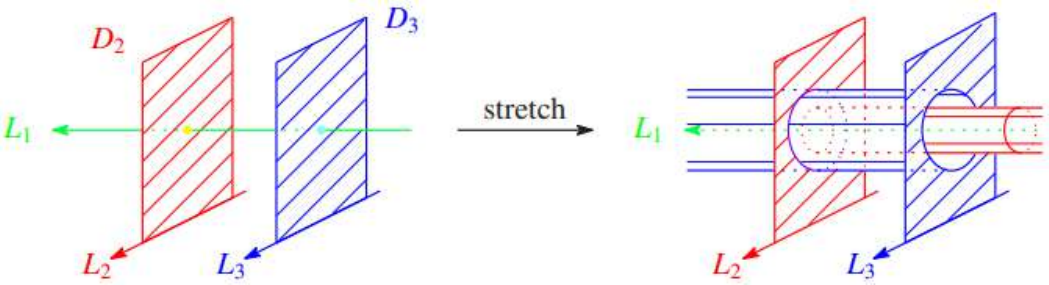}
    \caption{Moving the intersection of a surface and the link by stretching}
    \label{figure_stretch_surface_bbcc23.TEX}
  \end{figure}
  Now, we compute $m_{123}(\Sigma^1\cup\Sigma^2\cup\Sigma^3)$. 
  The word $w_1$ of the intersection points of $\Sigma^2$ with $L_1$ and the intersection of $\Sigma^3$ with $L_1$, by the choice of $b_{nc_2}$ and $b_{c_3}$, is given by
  \begin{equation}
    w_1=3^{\epsilon_1}\cdots3^{\epsilon_i}2^{\epsilon_{i+1}}\cdots2^{\epsilon_{i^\prime}}, \quad\epsilon_j\in\{\pm1\}. 
  \end{equation}
  In this case, $e_{23}(w_1)=0$. 
  Similarly, $e_{31}(w_2)=0$ and $e_{12}(w_3)=0$, depending on the choice of appropriate base points. 
  Therefore, 
  \begin{equation}\label{m123Sigmaprime}
    m_{123}(\Sigma^1\cup\Sigma^2\cup\Sigma^3)\equiv0\mod\Delta_L(123). 
  \end{equation}
  Next, we compute $t_{123}(\Sigma^1\cup\Sigma^2\cup\Sigma^3)$. 
  There are two types of triple points on the surface: those that already existed in $D_1\cup D_2\cup D_3$, and those created by the stretching during the construction of $\Sigma^1\cup\Sigma^2\cup\Sigma^3$. 
  The triple points that were present in $D_1\cup D_2\cup D_3$ remain unchanged in $\Sigma^1\cup\Sigma^2\cup\Sigma^3$. 
  By the proof of Theorem \ref{IDthm3}, we have
  \begin{equation}\label{t123D1D2D3}
    t_{123}(D_1\cup D_2\cup D_3)=\domu{1}{2}{3}. 
  \end{equation}
  Let $\alpha_k\, (k=1, 2, 3)$ denote the sum of the signs of the newly created triple points during the stretching process along $L_k$. 
  Then, we have
  \begin{equation}\label{t123Sigmaprime}
    t_{123}(\Sigma^1\cup\Sigma^2\cup\Sigma^3)=\domu{1}{2}{3}+\sum_{k=1}^3\alpha_k. 
  \end{equation}
  We will now show that
  \begin{equation}\label{alphaequolmarubatu}
    \alpha_3=\left\langle\marubatu{3}{2}, G_{L_1}\right\rangle. 
  \end{equation}

  If we write $\eL{k}{a}{b}$, it represents an arc cut along the direction of $L_k$, with the starting point at $a$ and the endpoint at $b$. 
  First, we observe where the triple points arise in the stretching process (see Figure \ref{figure_Thmproof1.tex}). 
  \begin{figure}[htbp]
    \centering
    \includegraphics[width=180mm]{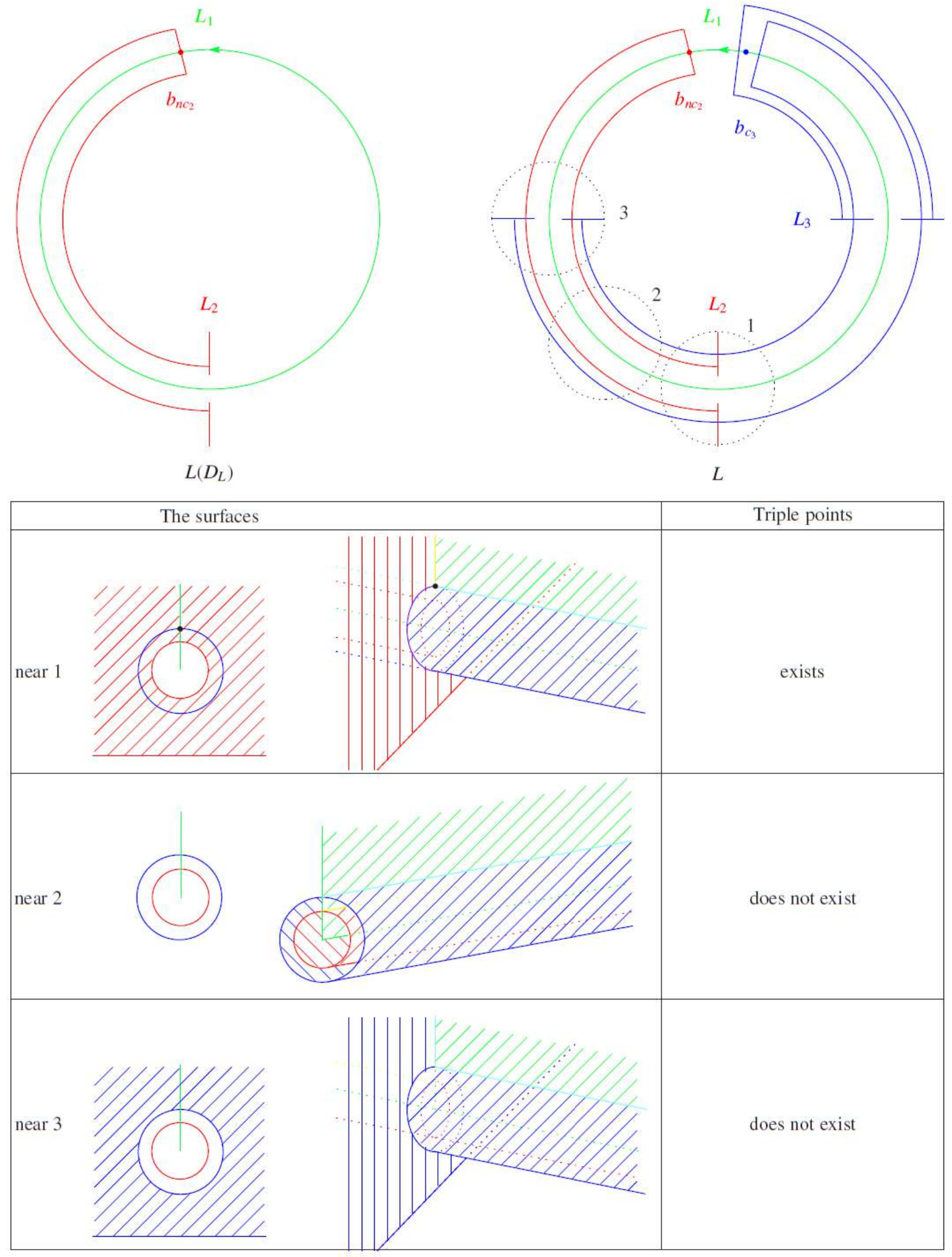}
    \caption{Illustration of where triple points occur}
    \label{figure_Thmproof1.tex}
  \end{figure}
  From Figure \ref{figure_Thmproof1.tex}, we can see that triple points are created in the neighborhoods of the crossings in $\vla{1}{2}$, where the stretched part of $\Sigma^3$ and the surface $\Sigma^2$ collide (near 1 in Figure \ref{figure_Thmproof1.tex}). 
  This is in one-to-one correspondence with the elements of $\vla{1}{3}$ on the arc $\eL{1}{b_{nc_2}}{p}$ (see Figure \ref{figure_Thmproof23.tex}). 
  \begin{figure}[htbp]
    \centering
    \includegraphics[width=150mm]{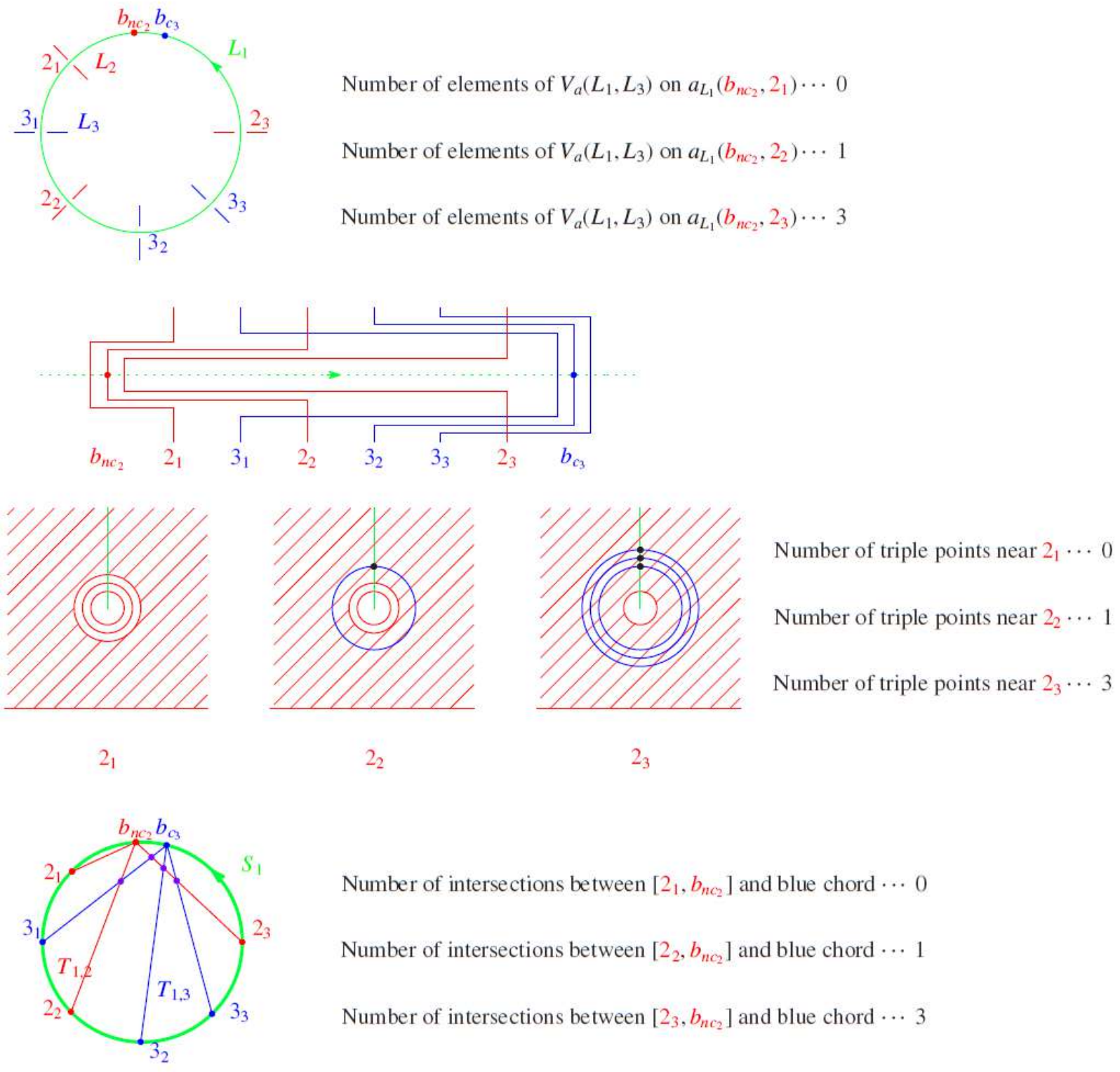}
    \caption{Correspondence between triple points and intersections of chord diagrams}
    \label{figure_Thmproof23.tex}
  \end{figure}
  These triple points are also in one-to-one correspondence with the intersections of the chord connecting the points of $\vla{1}{2}$ and $b_{nc_2}$ in $G_{L_1}$ (see Figure \ref{figure_Thmproof23.tex}). 
  We can verify that the sign of such a triple point is equal to that of the corresponding intersection of the chord (see Figure \ref{figure_Thmproof41.tex}). 
  \begin{figure}[htbp]
    \centering
    \includegraphics[width=150mm]{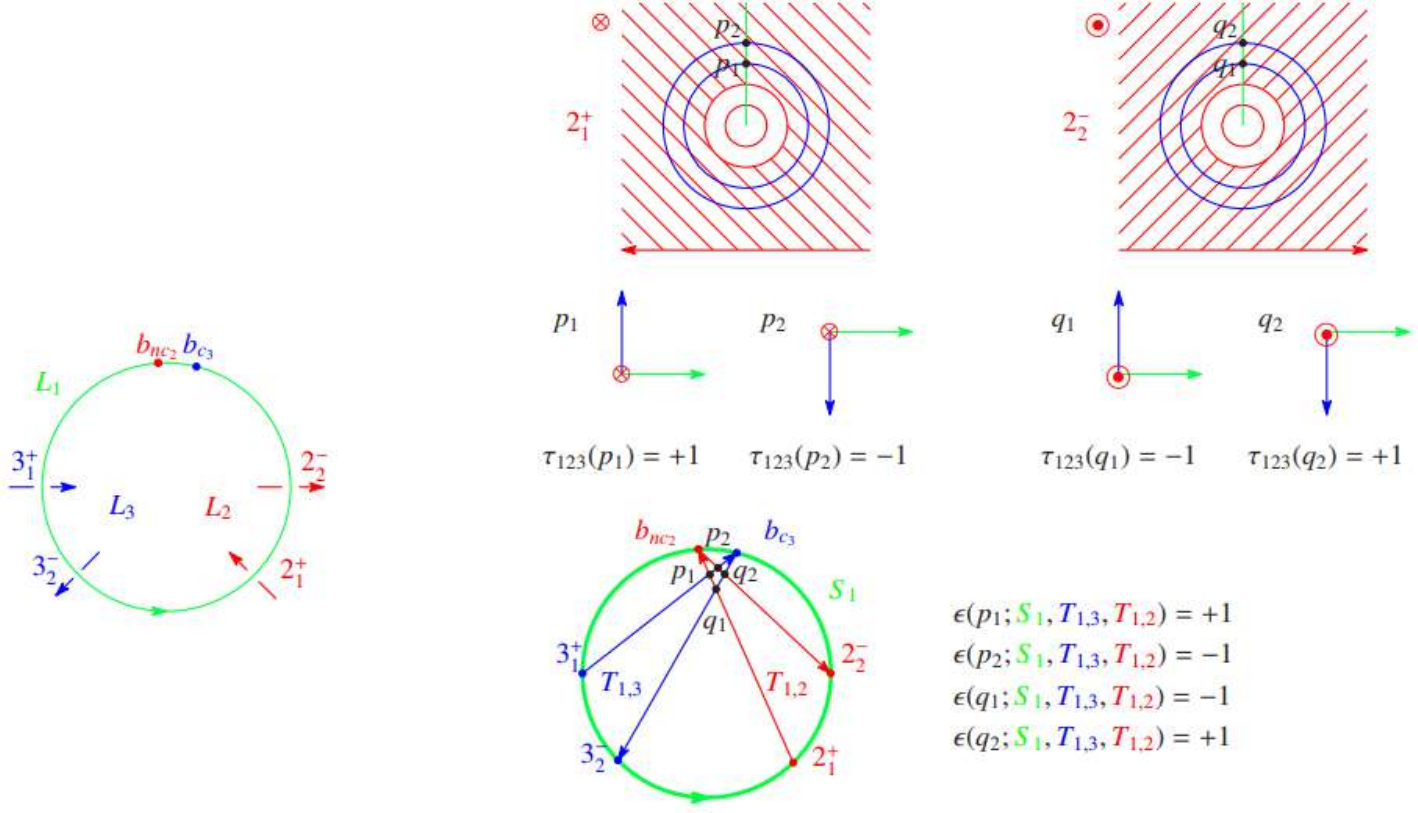}
    \caption{Correspondence between the signs of triple points and intersection signs in chord diagrams}
    \label{figure_Thmproof41.tex}
  \end{figure}
  For example, in Figure \ref{figure_Thmproof41.tex}, the surface stretched from $3_1^+$ collides with the surface near $2_1^+$ to form a triple point $p_1$, with the sign of $\tau_{123}(p_1)=+1$. 
  On the other hand, the chord diagram shows that the chord with the endpoint at $3_1^+$ intersects with the chord that and the endpoint at $2_1^+$, and the sign of $\dokhplus{p_1}{S_1}{T_{1, 3}}{T_{1, 2}}=+1$. 
  Thus \eqref{alphaequolmarubatu} holds. 
  The same reasoning applies to $\alpha_1=\left\langle\marubatu{1}{3}, G_{L_2}\right\rangle$ and $\alpha_2=\left\langle\marubatu{2}{1}, G_{L_3}\right\rangle$. 
  These results, along with equations \eqref{m123Sigmaprime} and \eqref{t123Sigmaprime} complete the proof. 
\end{proof}

  
\end{document}